\documentclass[12pt,reqno]{amsart}
\usepackage{amssymb,amsmath,amsthm}
\usepackage{fullpage}
\usepackage{enumerate}
\usepackage{xcolor}
\usepackage{mdframed}

\newtheorem{theorem}{Theorem}
\newtheorem{lemma}[theorem]{Lemma}

\newcommand{\bone}{\textbf{1}}
\newcommand{\hw}{\mbox{wt}}
\newcommand{\ahw}{\overline{\text{wt}}}

\newcommand{\bC}{\textbf{C}}

\begin{document}

\title{Improving Uniquely Decodable Codes in Binary Adder Channels}


\author{
József Balogh$^{1}$,
The Nguyen$^{1}$,
Patric R.J. Östergård$^{2}$,
Ethan P. White$^{1}$,
and Michael C. Wigal$^{1}$
}
\thanks{Balogh was supported in part by NSF grants DMS-1764123 and RTG DMS-1937241, FRG DMS-2152488, the Arnold O. Beckman Research Award (UIUC Campus Research Board RB 24012), the Langan Scholar Fund (UIUC). Nguyen is supported by David G. Bourgin Mathematics Fellowship. White is supported in part by an NSERC PDF. Wigal is supported in part by NSF  RTG DMS-1937241.\\
\texttt{\{jobal, thevn2, epw, wigal\}@illinois.edu}, \quad \texttt{patric.ostergard@aalto.fi}}

\maketitle

\begin{center}
\vspace{-1em}
\small
$^{1}$Department of Mathematics, University of Illinois Urbana–Champaign, Urbana IL 61801, USA \\
$^{2}$Department of Information and Communications Engineering, Aalto University School of Electrical Engineering, P.O. Box 15600, 00076 Aalto, Finland \\[6pt]
\end{center}

\begin{abstract}
Uniquely decodable codes for the $T$-user binary adder channel can be directly extended to
longer codes with the same sum rate. It is shown that if at least one of the original
constituent codes does not have average weight exactly half of the dimension, 
then there exists a code in a higher dimension with a strictly higher sum rate. Using 
this method, the best known lower bounds on the sum rate of codes for the $T$-user 
binary adder channel are improved for all $T \geq 2$. The case $T=2$ was coined 
the co-Sidon problem by Lindstr{\"o}m in his seminal work in the 1960s and is also known 
as the multi-set union-free problem. 
\end{abstract}

\section{Introduction} 

The $T$-user \emph{binary adder channel} is a communication system where $T$ users are synchronously sending messages encoded as binary vectors over a common channel. User $i$, $1 \leq i \leq T$ sends messages belonging to a  \emph{constituent code} $C_i \subset \{0,1\}^d$. The messages of all users are transmitted as a sum over $\mathbb{Z}^d$. The code $C_1,\ldots,C_T$ is \emph{uniquely decodable} (UD) if all transmitted sums are unique, that is, for every $x \in \mathbb{Z}^d$ there is at most one solution to
\[ u_1 + u_2 + \ldots + u_T = x,\]
where $u_i \in C_i$ and addition occurs over $\mathbb{Z}^d$. With a UD code the total number of messages that can be sent is $V = \prod_{i=1}^T |C_i|$, and the amount of
information that can be transmitted per symbol is $\log_2 V^{1/d}$ bits, which is defined
as the \emph{sum rate} of the code. 

For a UD code, one may consider the tuple of rates of the constituent codes.
The set of all possible such tuples of UD codes form a region in the $T$-dimensional Euclidean 
space called the \emph{zero-error capacity region} of the channel. The zero-error
capacity region of the 2-user binary adder channel has been thoroughly studied
\cite{AB,AKKN,CT,KLWY,K,MO,T,UL,W}. For $T \geq 3$, it is a complex task to study 
and visualize the zero-error capacity region, so the main focus has been on 
determining bounds on the maximum sum rate. 

The best known upper 
bound on the sum rate of codes for the $T$-user binary adder channel comes from an 
entropy argument of Lindstr\"om~\cite{L1}, which was later extended by Liao~\cite{HJL}:
\begin{equation}\label{upperBound} R_{\text{upper}} = \sum_{k=0}^T \frac{\binom{T}{k}}{2^T} \log_2 \frac{2^T}{\binom{T}{k}}. 
\end{equation}

Throughout the paper, approximate lower bounds are rounded down and approximate upper bounds are rounded up to get valid bounds. As for specific bounds on the sum rate of codes for the $T$-user binary adder channel, Lindstr\"om~\cite{L1} observed that the constituent codes 
$C_1 = \{01, 10, 11\}$ and $C_2 = \{00, 11\}$ give a lower bound 
$\log_2 \sqrt{6}\approx 1.2924$ on the maximum sum rate,
which can be compared with the upper bound of 
$\log_2 \sqrt{8} = 1.5$ from~\eqref{upperBound}. The current best known
lower bounds for $T=2$, $3 \leq T \leq 5$, and $T \geq 6$ have been obtained
in \cite{MO}, \cite{KO}, and \cite{KM}, respectively. Best known lower 
and upper bounds for $T \leq 8$ are, respectively, given in columns $R_{\text{old}}$ and 
$R_{\text{upper}}$ of Table~\ref{boundtable}.
For the lower bounds, the dimension
of the corresponding code is also given, in column $d_{\text{old}}$. Finally,
there are two columns for the dimension ($d_{\text{new}}$) and sum rate
($R_{\text{new}}$) of the codes constructed in the present work. 

\begin{table}[htbp] 
\centering
\vspace*{5mm}
 \begin{tabular}{| c |c | c| c | c | c | } 
 \hline
$T$ & $d_{\text{old}}$ & $R_{\text{old}}$ & $d_{\text{new}}$ & $R_{\text{new}}$ & $R_{\text{upper}}$ \\ 
\hline
\hline
2 & 6 & 1.3178 & 852 & 1.3184 & 1.5000\\ 
3 & 6 & 1.5381 & 432 & 1.5395 & 1.8113\\ 
4 & 4 & 1.7500 & 980 & 1.7505 & 2.0307\\ 
5 & 4 & 1.8962 & 896 & 1.8970 & 2.1982\\ 
6 & 4 & 2.0000 & 104 & 2.0052 & 2.3334\\ 
7 & 8 & 2.0731 & 128 & 2.0774 & 2.4467\\ 
8 & 6 & 2.1666 & 384 & 2.1683 & 2.5442\\
\hline
\end{tabular}
\vspace*{5mm}
\caption{Bounds on the maximum sum rate of $T$-user UD codes.}
\label{boundtable}
\end{table}

The binary adder channel has also been studied under other names and in
different contexts. In early work,
Lindstr\"om~\cite{L1} calls a UD code for the 2-user binary adder channel 
a \emph{co-Sidon pair}. In the same paper, he introduces a version where 
the codewords are not taken from different constituent codes but are taken 
from one single code known as a \emph{binary Sidon set} or a \emph{$B_2$-set}.
For bounds on the maximum size of $B_2$-sets, see \cite{CLZ,L2}.

Another perspective comes from interpreting elements of $\{0,1\}^d$ as subsets 
of a $d$-element set. A $d$-dimensional UD code for the 2-user binary adder channel 
then corresponds to two families of subsets,  
$\mathcal{F}_1, \mathcal{F}_2 \subset 2^{\{1,2,\ldots,d\}}$, with the property that the multisets $A \uplus B$ are all distinct for 
$A \in \mathcal{F}_1$ and $B \in \mathcal{F}_2$. This problem is studied in
\cite{OO} under the additional condition that $\log_2 |\mathcal{F}_1| = d(1-o(1))$.  
A variant of this problem is that of finding the maximum number of subsets of 
$\{1,2,\ldots ,d\}$ such that the pairwise union of two sets are all distinct \cite{FF}, 
that is, without allowing multisets.

Uniquely decodable codes for the $T$-user binary adder channel can directly be 
extended to longer codes with the same sum rate. In this work, it is shown that
if at least one of the original constituent codes does not have average weight 
exactly half of the dimension, then there is a code in a higher dimension with 
a strictly higher sum rate. This result makes it possible to 
improve the previously best-known lower bound on the zero-error capacity of the $T$-user
binary adder channel for all $T \geq 2$.

The paper is organized as follows. In Section~\ref{sect:construction} it
is formally proved that constituent codes that do not have average weight 
half of the dimension can be used to construct codes with higher sum rate.
A constructive practical method for codes with small parameters is described in 
Section \ref{sect:codes} and used to obtain codes that give the new bounds
in Table \ref{boundtable}. The paper is concluded in Section~\ref{sect:con}.


\section{A General Existence Proof} 
\label{sect:construction}

Let us first introduce some concepts and notation. Two codes for the $T$-user binary adder channel are said to be \emph{equivalent} if one can be obtained from the other by a permutation of the coordinates followed by negating the coordinate values 
in a subset of coordinates. Such transformations are carried out simultaneously in all constituent codes. Code equivalence preserves the property of being a UD code. By negating coordinate values exactly in those coordinates where there are more 1s than 0s, the average vector weight of a constituent code is minimized within 
an equivalence class. Such minimization will be a core part of obtaining the improvements 
in the present work. 

Denote by $\hw(\cdot)$  the {\it Hamming weight} of a vector, that is, the sum of the absolute values of its entries. Define $\ahw(C)$ as the average Hamming weight of all vectors in $C$. Denote by $\bone^d$ the vector of length $d$ with $1$ in all coordinates, and $\bone^d - C := \{ \bone^d - x \colon x \in C\}$. 
For a subset $C \subset \{0,1\}^d$ and a positive integer $n \geq 1$, define 
\begin{equation*}\label{tensorDef} C^{ n} := \{ (v_1,v_2,\ldots,v_n) \colon v_i \in C,\  1 \leq i \leq n\} \subset \{0,1\}^{dn}.\end{equation*}

Our construction can now be sketched as follows. If $\bC = C_1,\ldots,C_T$ is a UD code in $\{0,1\}^d$ with $\ahw(C_1) < d/2$ and sum rate $R$, then both $\bC_0 = C_1^n, \ldots, C_T^n$ and $\bC_1 = (\bone^d - C_1)^n, C_2^n, \ldots, C_T^n$ are also UD codes in $\{0,1\}^{dn}$ with sum rate $R$ for all positive integers $n$. Our goal is to `glue' $\bC_0$ and $\bC_1$ together to create a new code with strictly higher rate. Taking the union of each constituent code may not result in a UD code, so we first refine both $\bC_0$ and $\bC_1$. To do so, we select a subset of vectors $C_i^* \subset C_i^n$ that are close to the average weight $\ahw(C_i^n) = n \cdot \ahw(C_i)$ for all $1 \leq i \leq T$. Observe $\bC_0^* = C_1^*, C_2^*, \ldots, C_T^*$ and $\bC_1^* = (\bone^{dn} - C_1^*), C_2^*, \ldots, C_T^*$ are also both UD codes. Moreover, the selection can be made so that 
\[\sum_{1 \leq i \leq T} \max_{c \in C_i^*}\mbox{wt}(c) < \min_{c\in (\bone^{dn}-C_1^*)}\mbox{wt}(c) + \sum_{2 \leq i \leq T} \min_{c \in C_i^*}\mbox{wt}(c),\]
which implies that $C_1^* \cup (\bone^{dn} - C_1^*), C_2^*, \ldots, C_T^*$ is also a UD code. The crux of our argument is that $\bC_0^*$ and $\bC_1^*$ can be distinguished by their Hamming weights while remaining large by concentration of measure. The upcoming lemma quantifies the former claim. For a nonempty binary code $C \subset \{0,1\}^n$, we define the variance as 
\[  \sigma^2 :=  \sigma^2(C) = \frac{1}{|C|} \sum_{c \in C} \left(\hw(c) - \ahw(C) \right)^2, \] 
and if $\sigma^2 > 0$, we define the third absolute standardized central moment of $C$ as the following,  
\[ \rho^3 :=  \rho^3(C) = \frac{1}{|C| } \sum_{c \in C} \left|\frac{\hw(c) - \ahw(C)}{\sigma}\right|^3.\] 
The constants $0.345$ and $0.69$ in the following lemma are derived from Korolev and Shevtsova's proof of the Berry-Essen Theorem \cite{KS}.

\begin{lemma}
\label{lemma}
Let $n \ge 1$ and let $C \subset \{0,1\}^d$ be a code
with variance $\sigma^2 > 0$ and third absolute standardized central moment $\rho^3 \ge 0$. For any $t > 0$, let 
\[ C^* = \{c \in C^n \, \colon |\hw(c) - \ahw(C^n)| \le t \}.\]
Then,
\[|C^*| \ge |C^n| \left(1 - \exp\left(\frac{-t^2}{2n\sigma^2}\right) - 0.69 \cdot  \frac{1+\rho^3}{\sqrt{n}}\right).\]
Furthermore,  let 
\begin{center}
    $C^*_{\ell} = \{c \in C^n \colon \hw(c) \ge  \ahw(C) - t \}$ and $ C^*_{u} = \{c \in C^n \colon \hw(c) \le \ahw(C) + t\}$.
\end{center}
Then for any $t > 0$,
\[ \min\{|C^*_{\ell}|,|C^*_{u}|\} \ge  |C^n| \left(1 - \frac{1}{2}\exp \left(\frac{-t^2}{2n\sigma^2}\right) - 0.345 \cdot \frac{1 + \rho^3}{\sqrt{n}}\right).\]
\end{lemma}

\begin{proof}
Let $X$ denote the random variable taking values in $[0,d]$ such that for every $k \in [0,d]$
\[ \mathbb{P}(X = k) = \frac{| \{  c \in C \colon \hw(c) = k \}|}{|C|}. \]
Note that $\mathbb{E}[X] = \ahw(C)$, and the variance and third absolute standardized central moment of $X$ are $\sigma^2$ and $\rho^3$ respectively. Let $X_1, \ldots, X_n$ denote i.i.d.~copies of $X$, then, 
In particular, 
\[|C^*| = |C^n| \cdot \mathbb{P} \left( \ahw(C^n) - t \le \sum_{i = 1}^n X_i \le \ahw(C^n) + t  \right).\]

Let $\Phi(x)$ be the cumulative distribution function of the standard normal and define $Y_n = \sigma^{-1} n^{-1/2}\sum_{i=1}^n \left(X_i - \ahw(C)\right)$, and let $F_n(y)$ be the cumulative distribution function of $Y_n$. Applying the Berry--Esseen theorem~\cite[Theorem 1]{KS} to $Y_n$ gives 
\begin{equation}\label{berry_esseen}
    | F_n(x) - \Phi(x)| \leq 0.345 \cdot  \frac{1+\rho^3}{\sqrt{n}}. 
\end{equation}
This directly implies, 
\begin{align}\label{concentration} \mathbb{P} \left( \ahw(C^n) - t \le \sum_{i = 1}^n X_i \le \ahw(C^n) + t  \right) = \mathbb{P} \left(\frac{-t}{\sigma\sqrt{n}} \leq  \sum_{i=1}^n \frac{X_i - \hw(C)}{\sigma\sqrt{n}} \leq \frac{t}{\sigma\sqrt{n}}  \right) \notag \\
\geq \Phi\left(\frac{t}{\sigma\sqrt{n}} \right) - \Phi\left(\frac{-t}{\sigma\sqrt{n}} \right) - 0.69 \cdot  \frac{1+\rho^3}{\sqrt{n}}. 
\end{align}

 The following tail bound can be derived with elementary calculus,
\begin{equation}\label{CNest}
\Phi(x) \leq \frac{1}{2} e^{-x^2/2} \quad \text{for all} \quad x \le 0. 
\end{equation}
 By~\eqref{concentration}~and~\eqref{CNest} we may conclude,
\begin{displaymath}
\mathbb{P} \left(\ahw(C^n)- t \leq  \sum_{i=1}^n X_i \leq \ahw(C^n)+ t \right) \geq 1 - \exp\left(\frac{-t^2}{2n\sigma^2}\right) - 0.69 \cdot  \frac{1+\rho^3}{\sqrt{n}}.
\end{displaymath}

 Let $C_{\ell}^*$ and $C_{u}^*$ be defined as in the lemma statement. 
The one-sided analogs are proved in a similar manner. We only sketch the lower bound for $|C_{\ell}^*|$ as the argument for $|C_u^*|$ follows by symmetry. By \eqref{berry_esseen} and \eqref{CNest},
\begin{align*}
    |C_{\ell}^*| &= |C^n| \cdot \mathbb{P}\left( \sum_{i = 1}^n X_i  \ge \ahw(C) - t \right) = |C^n| \cdot \mathbb{P}\left( \sum_{i = 1}^n \left(\frac{X_i - \ahw(C)}{\sigma\sqrt{n}} \right) \ge \frac{- t}{\sigma\sqrt{n}} \right) \\
    &\ge |C^n| \left(\Phi\left(\frac{-t}{\sigma\sqrt{n}} \right) - 0.345 \cdot \frac{1 + \rho^3}{\sqrt{n}}\right) = |C^n| \left(1 - \frac{1}{2}\exp \left(\frac{-t^2}{2n\sigma^2}\right) - 0.345 \cdot \frac{1 + \rho^3}{\sqrt{n}}\right).
\end{align*}

\end{proof}

\begin{theorem}\label{mainT}
Let $C_1,\ldots,C_T \subset \{ 0,1\}^d$ be the constituent codes of a $T$-user UD code with rate $R$. Suppose that there exists a $j$ such that $\ahw(C_j) \neq d/2$. Then there exist a positive integer $n$ and constituent codes $C_j^* \subset C_j^n$ such that $C_1^*,\ldots,C_T^*$ is a $T$-user UD code with rate strictly larger than $R$.
\end{theorem}

\begin{proof} Without loss of generality, we may consider the case $j = 1$ and $\ahw(C_1) < d/2$ by permuting coordinates and negating values. For all $i$, let $\sigma_i^2$ denote the variance of $C_i$. If $\sigma_i^2 > 0$, let $\rho_i^3$ denote the third absolute standardized central moment of $C_i$. Let $n$ be a positive integer to be chosen later. Define the subset of indices $I \subset [T]$ such that
\begin{center}
    $ I = \{i \in [T] \colon \sigma_i^2 > 0\}$.
\end{center}

When $I \neq \emptyset$, define 
\[ \beta = 0.345 \cdot \max_{i \in I} (1+\rho_{i}^3), \]
and when $I \setminus \{ 1 \} \neq \emptyset$ define
\[ \alpha = \frac{\frac{d}{2} - \ahw(C_1)}{2 \sum_{i = 2}^{T} \sigma_{i}}. \]
We first define $C_i^*$ for $i \ge 2$. If $i \not \in I$, we set $C_i^* = C_i^n$. Otherwise, if $i \in I$, we set
\[ C_i^* = \{c \in C_i^n \colon |\hw(c) - n\ahw(C_i)| \le \alpha n \sigma_i \}. \]
In the case $i \in I$, applying Lemma~\ref{lemma} with $t = \alpha n \sigma_i$, we have 
\[ |C_i^*| \geq |C_i|^n \left( 1 - \exp \left(-\frac{n\alpha^2}{2}\right) - 2 \beta n^{-1/2} \right). \]

Let $A = C_1^n$ and $B = (\bone^d - C_1)^n$. Observe that $A, C_2^*, \ldots, C_{T}^*$ and $B, C_2^*, \ldots, C_{T}^*$ are each $T$-user UD codes. Let \[\kappa = \frac{1}{2} \left(\frac{d}{2} - \ahw(C_1) \right).\] Note that the following inequality holds whenever $I \setminus \{ 1 \} \neq \emptyset$.
\begin{equation}\label{kappaBound} \kappa \ge \alpha \sum_{i = 2}^{T} \sigma_i.\end{equation}

Define $C_1^* = A^* \cup B^*$ where 
\begin{align*}
    A^*  &= \{c \in C_1^n \colon \hw(c) \in [0, dn/2 - \kappa n) \}; \text{ and}\\
    B^*  &= \{c \in (\bone^d - C_1)^n \colon \hw(c) \in [ dn/2+ \kappa n, dn] \}.
\end{align*}
Clearly, $A^*, C_2^*,\ldots,C_{T}^*$ and $B^*, C_2^*,\ldots,C_{T}^*$ are both UD codes. Let $u$ be a vector in  $A^* + C_2^*+\ldots+C_{T}^*$ and $v$ be a vector in $B^* + C_2^*+\ldots+C_{T}^*$. Then 

\begin{align*} 
 \hw(u) &< dn/2 -\kappa n + \sum_{i = 2}^{T} \ahw(C_i) + \alpha n\sum_{i = 2}^{T}  \sigma_i; \text{ and}  \\
 \hw(v) &\ge dn/2 + \kappa n + \sum_{i = 2}^{T} \ahw(C_i) - \alpha n\sum_{i = 2}^{T} \sigma_i.
\end{align*}
By~\eqref{kappaBound} we have $\kappa n \ge \alpha n \sum_{i = 2}^{T}  \sigma_i$, and therefore $\hw(u) < \hw(v)$. In particular, every vector in $A^* + C_2^*+\ldots+C_{T}^*$ has a Hamming weight strictly less than $B^* + C_2^*+\ldots+C_{T}^*$.  Thus, we may conclude that $C_1^*,\ldots,C_T^*$ is also a UD code. To finish the proof, it suffices to show that $|C_1^*|$ is large. Observe that
\begin{align*} \frac{dn}{2} - \kappa n &= n\ahw(C_1) + \frac{n}{2}\left(\frac{d}{2} - \ahw(C_1)\right) \\ 
\frac{dn}{2} + \kappa n &= n\left( d - \ahw(C_1) \right) - \frac{n}{2}\left(\frac{d}{2} - \ahw(C_1)\right) .\end{align*}

Hence, when $\sigma_1 = 0$, the Hamming weight constraints in the definitions of $A^*$ and $B^*$ are redundant and we have $|A^*| = |C_1|^n = |B^*|$. Otherwise, by Lemma~\ref{lemma} with $t = n\left(d/2 - \ahw(C_1)\right)/2$, we have 
\[     \min \{ |A^*|, |B^*| \}  \geq  |C_1|^n \left( 1 - \frac{1}{2}\exp \left(-\frac{n\left(d/2 - \ahw(C_1)\right)^2}{8\sigma_1^2}\right) - \beta n^{-1/2}  \right).\] 
 Define 
\[\theta = \max \left\{ \exp \left(-n\alpha^2/2\right) + 2 \beta n^{-1/2}, \, \frac{1}{2}\exp \left(-\frac{n}{8\sigma_1^2}\left(\frac{d}{2} - \ahw(C_1)\right)^2\right) + \beta n^{-1/2} \right\}.\]

The sum rate of the new code is 
\[ \frac{1}{dn} \log_2 \prod_{i=1}^T |C_i^*| \geq \frac{1}{dn}   \log_2  \left(2\prod_{i=1}^T  |C_i|^n (1-\theta)\right) = R+ \frac{1}{dn } \log_2 \left(2(1-\theta)^T\right).    \]

Since $\theta \to 0$ as $n \to \infty$, for sufficiently large $n$, the above is strictly greater than $R$. 
\end{proof}

In the proof of Theorem~\ref{mainT}, we showed that the sum rate of the new code is higher than $R$. However, the improvement term $\log_2(2(1-\theta)^T)/(dn)$ tends to 0 as $n \to \infty$, so the sum rate starts to decrease when $n$ becomes large enough. Details of this kind will matter when we consider the problem of constructing specific UD codes for small $T$.

\section{New Codes}
\label{sect:codes}

Suppose $C_1,\ldots,C_T \subset \{ 0,1\}^d$ are the constituent codes of a $T$-user UD code. 
Our construction of UD codes of higher sum rate follows three steps and is based on the 
ideas in the proof of Theorem~\ref{mainT}. It is straightforward to verify that the constructed codes are UD. Let $N$ be a fixed limit for the value of $n$ when 
constructing codes $C^n$.

\vspace{0.5cm}
\begin{mdframed}
\textbf{Step 1}: Negate the coordinate values in a subset of coordinates to produce 
an equivalent code $C_1',\ldots,C_T' \subset \{ 0,1\}^d$ such that $\min_i \ahw(C_i')$ 
is as small as possible. Permute the indices of the codes so that 
$\ahw(C_1') = \min_i \ahw(C_i')$.\\
\textbf{Step 2}: For $1 \leq i \leq T$ and $1 \leq n \leq N$, 
determine the weight distribution of $(C_i')^n$.\\
\textbf{Step 3}: 
Given values of $g_i$, $2 \leq i \leq T$, let $g = \sum_{i=2}^T g_i$.
The new constituent codes $C_1^*,\ldots,C_T^*$ are
\begin{displaymath}\label{CiChoice} 
C_i^* = \{x \in (C_i')^n \colon \hw(x) \in [n\ahw(C_i')-g_i, n \ahw (C_i')+g_i] \}, 
\end{displaymath}
for $i\ge 2$ and $C_1^* = A^* \cup B^*$, where 
\begin{align*}
    A^* & = \{x \in (C_1')^n \colon \hw(x) \in [0, dn/2 - g -1] \},\\
    B^* & = \{x \in (\bone^d - C_1')^n \colon \hw(x) \in [dn/2+ g, dn] \},
\end{align*}
Carry out this step in an exhaustive manner for different values of $n$ and 
$g_i, 2 \leq i \leq T$. 
\end{mdframed}

\vspace{0.2cm}

Before presenting the results obtained by this approach, we give some remarks on our process of finding improved codes. There may be several non-equivalent best known codes for some $T$, and there is not necessarily a unique way of carrying out Step 1. All
such possibilities should be tried. We use dynamic programming to determine the weight distribution of $(C_i')^n$ from that of $(C_i')^{n-1}$. It would be impractical to construct the codes in Step 3, but we only need to know their sizes, which we get by using the results of Step 2. When $T$ is large, it is useful to reduce the search space of the values of $g_i$. This can be done by considering symmetries and searching in a subset of all permissible values. We need $A^* + C_2^*+\ldots+C_{T}^*$ and $B^*+C_2^*+\ldots+C_{T}^*$ to have an empty intersection. This property is achieved by constraining the Hamming weight of the elements of $A^*$ and $B^*$. 

To obtain the new results, we have used known codes and also constructed codes that are not equivalent to known ones. To search for codes, we fixed the sizes of the constituent codes and used tabu search \cite{HO}. Binary codewords are displayed in decimal format. 

\subsection{Two users}\label{2usersec}

The following code with $d=6$ was found in the work published in \cite{MO} but is not explicitly listed there.
\vspace{2mm}
\begin{align*}
C_1 &= \{3, 4, 7, 10, 14, 17, 21, 27, 32, 36, 42, 49, 56, 59, 60\},\\
C_2 &= \{8, 9, 16, 18, 24, 29, 30, 31, 32, 33, 34, 39, 45, 47, 54, 55\}.
\end{align*}

\vspace{2mm}
\noindent 
We have $\ahw(C_1) = 41/15$ and $\ahw(C_2) = 3$. The highest sum rate found by our algorithm occurs when $n = 142$, $g_2 = 24$ and is approximately $1.318446971$.

\subsection{Three users}

The code
\vspace{2mm}
\begin{align*}
C_1 &= \{3, 4, 7, 24, 27, 28, 32, 35, 36, 56\},\\
C_2 &= \{0, 2, 12, 14, 16, 18, 30, 33, 45, 47, 49, 51, 56, 61, 63\},\\
C_3 &= \{9, 21, 42, 54\}
\end{align*}

\vspace{2mm}
\noindent
with $d=6$ has $\ahw(C_1) = 12/5$, $\ahw(C_2) = \ahw(C_3) =  3$. The highest sum rate found by our algorithm occurs when $n = 72$, $g_2 = 21$, $g_3 = 11$ and is approximately $1.539572454$.

\subsection{Four users}

We know \cite{KO} that there are exactly two equivalence classes of codes 
with $T=4$, $d=4$, and constituent codes of size 2, 4, 4, and 4.
The following code is not listed in \cite{KO}.
\vspace{2mm}
\begin{equation*}
    C_1  = \{0,7,8,14\}, \quad C_2 = \{4,5,10,11\}, \quad  C_3 = \{2,6,9,13\}, \quad C_4 = \{3,12\}.
\end{equation*}

\vspace{2mm}
\noindent 
This code
has $\ahw(C_1) = 7/4$ and $\ahw(C_i) = 2$ for $2 \leq i \leq 4$. The weight distributions of 
$C_2$ and $C_3$ are the same, so we constrain $g_2 = g_3$.
Both vectors in $C_4$ have weight $2$, so we constrain $g_4 = 0$. With these restrictions, the highest sum rate obtained by our algorithm occurs when $n = 245$, $g_2 = g_3 = 20$, $g_4 = 0$ and is approximately $1.750590009$.   

\subsection{Five users}

The following code has the same sum rate and dimension ($d=4$) as the best known code for $T=5$ but it has a different distribution of the sizes of the constituent codes:
\vspace{2mm}
\begin{equation*}
    C_1  = \{0,4,11\}, \;\; C_2 = \{0,3,5,6,9,10,12,15\}, \;\;
    C_3 = \{1,14\}, \;\; C_4 = \{2,13\}, \;\; C_5 = \{7,8\}.
\end{equation*}

\vspace{2mm}
\noindent 
This code has $\ahw(C_1) = 4/3$, and $\ahw(C_i) = 2$ for $2 \leq i \leq 5$. 
The weight distribution of $C_3$, $C_4$, and $C_5$ are the same, so we constrain $g_3 = g_4 = g_5$. The highest sum rate obtained by our algorithm occurs when $n = 224$, $g_2 = g_3 = g_4 = g_5 = 30$ and is approximately $1.897008675$. 

\subsection{Six users} 

Prior to this work, for $T \geq 6$ the $T$-user UD codes 
with highest sum rate were those in \cite{KM}. We use such codes for $T = 6,7,8$ below. 
One code with $d=4$ is
\vspace{2mm}
\begin{align*} 
C_1 & = \{0,2,8,10\}, \quad C_2 = \{1,2,13,14\}, \quad C_3 = \{0,15\}, \\
C_4 & = \{5,10\}, \quad C_5 = \{3,12\}, \quad C_6 = \{6,9\}.
\end{align*}

\vspace{2mm}
\noindent 
We have $\ahw(C_1) = 1$ and $\ahw(C_i) = 2$ for $2 \leq i \leq 6$. All vectors in $C_4$, $C_5$, and $C_6$ have weight 2, so we constrain $g_4 = g_5 = g_6 = 0$. The highest sum rate found by our algorithm occurs when $n = 26$, $g_2 = 8$, $g_3 = 12$, $g_4 = g_5 = g_6 = 0$ and is approximately 
$2.005264438$.  

\subsection{Seven users}

The code
\vspace{2mm}
\begin{align*}
C_1 &= \{0, 2, 8, 10, 32, 34, 40, 42, 128, 130, 136, 138, 160, 162, 168, 170\}, \\  
C_2 &= \{17, 18, 29, 30, 33, 34, 45, 46, 209, 210, 221, 222, 225, 226, 237, 238\}, \\  
C_3 &= \{0, 240, 255\}, \quad C_4 = \{15, 240\}, \quad C_5 = \{85, 90, 165, 170\}, \\ 
C_6 &= \{51, 60, 195, 204\}, \quad C_7 = \{102, 105, 150, 153\}
\end{align*}

\vspace{2mm}
\noindent 
with $d=8$ has $\ahw(C_1) = 2$ and $\ahw(C_i) = 4$ for $2 \leq i \leq 7$. 
All vectors in $C_4$, $C_5$, $C_6$, and $C_7$ have weight $4$, so we constrain $g_4 = g_5 = g_6 = g_7 = 0$. The highest sum rate found by our algorithm occurs when $n = 16$, $g_2 = 10$, $g_3 = 16$, $g_4 = g_5 = g_6 = g_7 = 0$ and is approximately 
$2.077479836$.

\subsection{Eight users}

The code
\vspace{2mm}
\begin{align*}
C_1 &= \{0,8,16,24,32,40,48,56\},\quad C_2 = \{2,16,38,52\}, \quad 
C_3 = \{0,63\}, \quad C_4 = \{9,54\},\\
C_5 &= \{18,27,36,45\}, \quad \quad \quad \quad \quad \,\, 
C_6 = \{21,28,35,42\}, \quad C_7 = \{7,56\}, \quad C_8 = \{14,49\}
\end{align*}

\vspace{2mm}
\noindent 
with $d=6$ has $\ahw(C_1) = 3/2$, $\ahw(C_2) = 2$, and $\ahw(C_i) = 3$ for $3 \leq i \leq 8$. 
All vectors in $C_6$, $C_7$, and $C_8$ have weight 3, so we constrain $g_6 = g_7 = g_8 = 0$. The highest sum rate found by our algorithm occurs when $n = 69$, $g_2 = 17$, $g_3 = 39$, $g_4 = g_5 = 17$, $g_6 = g_7 = g_8 = 0$ and is approximately $2.168328140$. In this new code, $\ahw(C_2^*) = 2 \cdot 69$, which is less than half of the new dimension $6 \cdot 69$. Consequently, by iterating the technique, this bound can be improved. However, those computations seem infeasible and the improvement would be small.

\section{Conclusion}
\label{sect:con}

Our method of improving the sum rate of codes when a constituent code does not have average weight exactly half of the dimension has led to improvements in many cases where the number of users is small. As the numerical improvements are minor and the code sizes are gigantic, it is clear that the flavor of the contribution is theoretical rather than practical. \\

New bounds for the capacity (region) of classical communication
channels are rarely obtained, but here it is shown that there exist better codes than the known codes for the binary adder channel with up to eight users. We hope that the main idea of the construction could lead to other new results for this type of channel. More generally, in the literature there are many types of codes for other channels that can be extended to infinite families; perhaps there are other cases where some property of the seed code could lead to improvements for larger codes in the family.

\clearpage

\end{document}